\documentclass[11pt]{article}

\usepackage[a4paper,margin=26mm]{geometry}
\usepackage[T1]{fontenc}
\usepackage[utf8]{inputenc}
\usepackage{lmodern}
\usepackage{microtype}
\usepackage{amsmath,amssymb,amsthm,mathtools}
\usepackage{booktabs,longtable,array}
\usepackage{enumitem}
\usepackage{needspace}
\usepackage{xcolor}
\usepackage[hidelinks]{hyperref}
\usepackage{url}

\hypersetup{
  pdftitle={Fourteen and fifteen lonely runners},
  pdfauthor={Jaan Allikvere}
}

\newtheorem{theorem}{Theorem}[section]
\newtheorem{proposition}[theorem]{Proposition}
\newtheorem{lemma}[theorem]{Lemma}
\newtheorem{corollary}[theorem]{Corollary}
\theoremstyle{definition}
\newtheorem{definition}[theorem]{Definition}
\theoremstyle{remark}
\newtheorem{remark}[theorem]{Remark}

\newcommand{\Z}{\mathbb{Z}}
\newcommand{\R}{\mathbb{R}}
\newcommand{\vct}[1]{\mathbf{#1}}
\newcommand{\nearest}[1]{\left\lVert #1\right\rVert}
\newcommand{\Enorm}[1]{\left\lVert #1\right\rVert_{E}}
\DeclareMathOperator{\lcm}{lcm}
\DeclareMathOperator{\spn}{span}
\DeclareMathOperator{\covol}{covol}

\newcommand{\GateCountThirteen}{61}
\newcommand{\GateMassThirteen}{353.7725}
\newcommand{\GateCountFourteen}{71}
\newcommand{\GateMassFourteen}{408.8233}

\newcommand{\RawThreshUB}{414.7794}       
\newcommand{\DivisorLogLB}{12.7948}       
\newcommand{\TargetUB}{401.9846}          

\title{Fourteen and fifteen lonely runners}
\author{Jaan Allikvere\\
\small Independent researcher, Tallinn, Estonia\\
\small ORCID: \href{https://orcid.org/0009-0003-5228-7015}{0009-0003-5228-7015}}
\date{September 2026}

\begin{document}
\maketitle

\begin{abstract}
We prove the Lonely Runner Conjecture for fourteen and fifteen runners.
Our proof combines a stronger bound on the speed product in a primitive
counterexample with exhaustive computations modulo primes.  To obtain the
bound, we work with a projected lattice basis and use cases of the
conjecture with fewer runners to bound partial sums of its squared
Gram--Schmidt lengths.  For fifteen runners, the product bound derived
from the work of Malikiosis, Santos, and Schymura gives a logarithmic
threshold of about $810$; ours reduces this to about $415$, making the
computation feasible.  For fourteen runners, the new bound reduces the
required number of primes from $111$ to $61$.

The computations start with a complete two-branch covering search, followed
by binary lifting.  Since $14$ and $15$ are composite, the polynomial
argument used in earlier work does not apply directly.  We finish the
fourteen-runner case by a direct search.  For fifteen runners, we use the
factorisation $15=3\cdot5$ and a shifting argument when all but a few speeds
share a common divisor.  The code and certificates are publicly archived.
\end{abstract}

computer-assisted proof; exhaustive computation.

\section{Introduction}

For a real number $x$, let $\nearest{x}$ denote its distance to the nearest
integer.  The Lonely Runner Conjecture $LRC(k)$ asserts that for every
$\vct{v}=(v_1,\ldots,v_k)\in(\Z\setminus\{0\})^k$ there is a real time $t$
with
\begin{equation}\label{eq:lrc}
  \nearest{t v_i}\geq \frac{1}{k+1}\qquad(1\leq i\leq k).
\end{equation}
This is the statement for $k+1$ runners after one runner is made
stationary; since $\nearest{-x}=\nearest{x}$, the speeds may be taken
positive.  We call $\vct{v}$ a \emph{speed tuple}.  If such a time $t$
exists, the tuple has the \emph{LR property}, and $t$ is a \emph{witness}.

The problem goes back to work of Wills in the 1960s~\cite{Wills1967}.
Rosenfeld proved $LRC(7)$~\cite{Rosenfeld25}; Trakulthongchai proved
$LRC(8)$ and $LRC(9)$~\cite{Trakulthongchai26}, and Rosenfeld
independently proved $LRC(8)$~\cite{Rosenfeld26}; Sungkawichai and
Trakulthongchai proved $LRC(k)$ for $10\leq k\leq12$~\cite{ST26}.

\begin{theorem}\label{thm:main}
The Lonely Runner Conjecture holds for fourteen and fifteen runners.
\end{theorem}

The recent proofs bound the speed product in a primitive counterexample,
then use modular computations to force many primes to divide it.  Once
their product exceeds the finite-checking bound, no counterexample can
exist.  We use the same strategy.

Two costs limit this approach.  At $p=199$, the normalized $k=13$
level-one search space already contains
$\binom{110}{12}=3{,}517{,}322{,}746{,}798{,}575$ candidates.  For
$LRC(14)$, the earlier bound would require primes up to about $970$, while
the measured generation cost grows roughly as $p^6$.  We need both a smaller
search and a stronger bound.  The composite final levels also require a
replacement for the polynomial argument of~\cite{ST26}.

\medskip
\noindent\textbf{What is new.}
The first contribution is a stronger finite-checking bound.  The recent
computer-assisted proofs for eight to thirteen
runners~\cite{Rosenfeld25,Trakulthongchai26,Rosenfeld26,ST26}, as well as
the first version of this paper, use a product bound derived from
Malikiosis, Santos, and Schymura~\cite{MSS25}.  Assuming $LRC(m)$ for
$m<n$, their bound implies that a primitive counterexample with $n$
positive speeds satisfies $\sum_iv_i<\binom{n+1}{2}^{n-1}$, hence
$v_1\cdots v_n<(\binom{n+1}{2}^{n-1}/n)^n$.
Their argument, like the independent one of Giri and
Kravitz~\cite{GK26}, rests on a single inequality in the top dimension.

Theorem~\ref{thm:flag} uses $n-1$ inequalities, one for each initial
segment of a Korkine--Zolotarev basis: the lower cases of the conjecture
bound every partial sum of its squared Gram--Schmidt lengths.  We combine
these inequalities with an exact
covolume formula in a norm adapted to the speeds.  This nearly halves the
logarithmic bound in both cases (Table~\ref{tab:thresholds}).
Remark~\ref{rem:compare} explains how the improvement is obtained.  For
fourteen runners the new bound reduces the calculation from $111$ verified
primes to $61$.  For fifteen runners it gives
\[
  \log(v_1\cdots v_{14})<414.779\ldots
  \qquad\text{in place of}\qquad 810.074\ldots,
\]
and the divisibility $360360=\lcm(2,\ldots,15)\mid v_1\cdots v_{14}$ then
reduces the required logarithmic prime sum to $\TargetUB$.

\begin{table}[ht]\centering
\caption{Bounds on $\log(v_1\cdots v_n)$ for a primitive counterexample,
assuming $LRC(m)$ for $m<n$.  We compare the bound of~\cite{MSS25} as
used in~\cite{ST26}, the successive-minima argument of
Remark~\ref{rem:compare}, and Theorem~\ref{thm:flag}.  None includes
the forced divisor of Lemma~\ref{lem:divisor}.}
\label{tab:thresholds}
\begin{tabular}{@{}cccc@{}}\toprule
$n$ & $\log B_n$~\cite{MSS25} & successive minima & Theorem~\ref{thm:flag}\\
\midrule
$13$ & $670.35$ & $470.18$ & $341.03$\\
$14$ & $810.07$ & $570.46$ & $414.78$\\
\bottomrule
\end{tabular}
\end{table}

At level one, a tuple has no witness precisely when each allowed time is
ruled out by some speed.  We refine the earlier covering searches by
treating covers in which every speed is needed and extensions of smaller
covers separately.  We prove that each orbit has a representative with
enough times covered by just one speed.  This justifies pruning without
losing an orbit
(Proposition~\ref{prop:basecomplete}).  After binary lifting, we finish
level $14$ by branch and bound, and level $15$ by factorisation and a shift
lemma for speeds with a common divisor and few exceptions.
The prime-divisibility criterion is due to~\cite{ST26}.

\medskip
\noindent\textbf{Outline of the proof.}
After recalling the modular reduction, we prove the flag bound in
Section~\ref{sec:flag} and describe the computations in
Section~\ref{sec:pipeline}.  The verified primes complete the proof in
Section~\ref{sec:gateset}.  We then discuss verification and the persistent
orbits.

\section{Proof strategy and the prime-divisibility criterion}\label{sec:framework}

We use the definitions of Sungkawichai and Trakulthongchai~\cite{ST26}
and state the prime-divisibility criterion in the form needed for both
proofs.

Let $p$ be prime and $l$ a positive integer, and write
$\Z_{p,l}=\Z_{pl}\setminus p\Z_l$.  Thus a vector in $\Z_{p,l}^k$ has no
coordinate divisible by $p$.

Assuming the lower cases of the conjecture, either test below rules out a
primitive counterexample in a residue class: a witness on the chosen time
grid, or a divisor of the level common to all but one speed.

\Needspace{5\baselineskip}
\begin{definition}
A vector $\vct{v}\in\Z_{p,l}^k$ is \emph{$(k,p,l)$-proper} if at least one
of the following holds:
\begin{enumerate}[leftmargin=*,label=(\alph*)]
\item for some $i$,
  $\gcd(l,v_1,\ldots,\widehat{v_i},\ldots,v_k)>1$;
\item for some $t\in(lp)^{-1}\Z$,
  $\nearest{tv_i}\geq1/(k+1)$ for every $i$.
\end{enumerate}
The set of vectors for which neither condition holds is $I(k,p,l)$.
A vector $\vct{v}\in\Z_{p,1}^k$ is \emph{eventually $(k,p)$-proper} if
some level $l$ has no improper lift of $\vct{v}$.  Here a lift to level $l$
is a vector $\vct{w}\equiv\vct{v}\pmod p$ with coordinates in $\Z_{pl}$.
The set of vectors that are not eventually proper is $J(k,p)$.
\end{definition}

By Lemma~2.4
of~\cite{ST26}, $J(k,p)=\varnothing$ if and only if $I(k,p,l)=\varnothing$
for some $l$.  The equivalence uses the following monotonicity: a witness,
or a divisor from the gcd condition, remains valid at every multiple of
the level.  By Proposition~5.1 of~\cite{ST26}, permutations,
sign changes of coordinates, and multiplication by a unit modulo $p$
preserve eventual properness (for the unit case see also Remark~2.3
below).

For an integer $c\geq2$, the $c$-lifts of $\vct{v}\in\Z_{p,l}^k$ to level
$cl$ are the $c^k$ vectors $w_i=v_i+a_ilp$, $a_i\in\{0,\ldots,c-1\}$; these
vectors form the full lift fiber of $\vct{v}$.  For a level-$1$ vector $r$
and a level $l$ write
$F_l(r)=\{\vct{w}\in I(k,p,l):\vct{w}\equiv r\bmod p\}$.  The computation
retains exactly $F_l(r)$ for every starting vector $r$ and every level it
reaches.  Indeed, every improper vector at level $cl$ reduces to one at
level $l$, since properness persists at multiples of the level.
Enumerating all lifts of $F_l(r)$ and discarding exactly the proper ones
gives $F_{cl}(r)$, proving the assertion by induction from level one.  If
$F_l(r)=\varnothing$ at some level, $r$ is eventually proper, and so is its
unit orbit.

\begin{lemma}[Prime-divisibility criterion]\label{lem:gate}
Let $k\geq3$, assume $LRC(m)$ for all $m<k$, and let $\vct{v}\in\Z_{>0}^k$
be a primitive counterexample to $LRC(k)$ with pairwise distinct
coordinates.  Let $p$ be a prime.
\begin{enumerate}[leftmargin=*,label=(\roman*)]
\item If $J(k,p)=\varnothing$, then $p\mid v_1\cdots v_k$.
\item More generally, suppose that for every $\vct{u}\in\Z_{p,1}^k$
  either $\vct{u}$ is eventually proper, or there is a level $l$ such
  that every lift $\vct{w}$ of $\vct{u}$ to level $l$ is proper or
  satisfies the following: every integer vector congruent to $\vct{w}$
  modulo $lp$ with pairwise distinct positive coordinates has the
  lonely-runner property~\eqref{eq:lrc}.  Then $p\mid v_1\cdots v_k$.
\end{enumerate}
\end{lemma}

\begin{proof}
Part (i) is Lemma~2.2 of~\cite{ST26} combined with Lemma~2.4 there.  We
recall the mechanism, since (ii) extends it.  Suppose $p\nmid v_i$ for all
$i$ and let $\vct{u}$ be the reduction of $\vct{v}$ modulo $p$.  Take a
level $l$ as in the hypothesis and let $\vct{w}$ be the reduction of
$\vct{v}$ modulo $lp$, a lift of $\vct{u}$.  If $\vct{w}$ has a witness
$t\in(lp)^{-1}\Z$, then $\nearest{tv_i}=\nearest{tw_i}\geq1/(k+1)$ for
all $i$, so $\vct{v}$ is not a counterexample.  If
$g=\gcd(l,w_j:j\neq i)>1$, then $g$ divides $v_j$ for all $j\neq i$, and
$g\nmid v_i$ because $\vct{v}$ is primitive.  The $k-1$ speeds $v_j$,
$j\neq i$, take $m\leq k-1$ distinct values, so $LRC(m)$ gives $t_0$ with
$\nearest{t_0v_j}\geq1/(m+1)\geq1/k$ for all $j\neq i$.  The shifted
times $t_0+q/g$, $q=0,\ldots,g-1$, leave every $t_0v_j$ ($j\neq i$) fixed
modulo one and move $t_0v_i$ through the $g'$ equally spaced points of a
grid with $g'=g/\gcd(g,v_i)\geq2$.  Such a grid cannot lie inside an open
arc of length $2/(k+1)\leq\frac12$: the shortest closed arc containing the
grid has length $1-1/g'\geq\frac12$, and equality is harmless because the
bad arc is open.  Thus some shift has
$\nearest{(t_0+q/g)v_i}\geq1/(k+1)$; at that shift every coordinate is at
distance $\geq1/(k+1)$ and $\vct{v}$ is not a counterexample.  In case
(ii), the third alternative applies directly to
$\vct{v}$ itself, which is congruent to $\vct{w}$ modulo $lp$ and has
pairwise distinct positive coordinates.  In every case $\vct{v}$ has the
lonely-runner property, contradicting the assumption; so some $v_i$ is
divisible by $p$.
\end{proof}

Part~(i) will prove all the required divisibilities for $k=13$ and most of
them for $k=14$.  For the remaining primes, part~(ii) and
Lemma~\ref{lem:neargcd} handle the residue classes left after lifting.

\begin{remark}[Unit orbits]
At the levels used below, $p>15$ and $\gcd(l,p)=1$.  If
$\vct{v}\equiv a r\pmod p$ with $a\in\Z_p^\times$, the Chinese
remainder theorem gives $b>0$ with $b\equiv a^{-1}\pmod p$ and
$b\equiv1\pmod l$, since $\gcd(l,p)=1$.  Then $(b\vct{v}\bmod lp)$ lifts
$r$, while $b\vct{v}$ remains positive and pairwise distinct, has the same
divisibility by divisors of $l$, and is a counterexample if and only if
$\vct{v}$ is.  Thus one orbit representative suffices.
\end{remark}

\section{The lattice flag bound and the forced divisor}\label{sec:flag}

This section follows the projected-lattice viewpoint of Giri and
Kravitz~\cite{GK26}.  We first place an ellipsoid inside the projected cube
and compute the lattice covolume in the norm defined by that ellipsoid.
We then bound the same covolume from below, using lower-dimensional cases
of the conjecture to constrain the Gram--Schmidt lengths of a
Korkine--Zolotarev basis.

Throughout this section $n\geq2$, $\vct{v}=(v_1,\ldots,v_n)\in\Z_{>0}^n$ is
primitive ($\gcd(v_1,\ldots,v_n)=1$), and ``counterexample'' means: for
every real $t$ some $i$ has $\nearest{tv_i}<1/(n+1)$.  We write
\[
  S=\sum_iv_i,\quad x_i=\frac{v_i}{S},\quad
  u_i=\frac{v_i}{\max_jv_j},\quad
  q_i=1+2u_i-u_i^2\in(1,2],\quad Q=\operatorname{diag}(q_1,\ldots,q_n).
\]
\Needspace{7\baselineskip}
Let $H=\vct{v}^{\perp}\subset\R^n$, let $P$ be the orthogonal projection
onto $H$, and let
\[
  \Lambda=P\Z^n,\qquad D=P[-1,1]^n,\qquad E=P\,Q^{1/2}B_2^n,
\]
\nopagebreak
where $B_2^n$ is the Euclidean unit ball.  $\Lambda$ is a lattice of rank
$d=n-1$ in $H$, with Euclidean covolume $\det\Lambda=1/\lVert\vct{v}\rVert_2$
(for primitive $\vct{v}$ the lattice $\Z^n\cap H$ has covolume
$\lVert\vct{v}\rVert_2$, and $P\Z^n$ is its dual in $H$).  $D$ is a zonotope
and $E$ is an origin-symmetric ellipsoid; we write $\Enorm{\cdot}$ for the
Euclidean norm on $H$ whose unit ball is $E$ and
$\langle\cdot,\cdot\rangle_E$ for its inner product.  Finally put
\begin{equation}\label{eq:F}
  F_n(\vct{x})^2=\prod_iq_i\cdot\sum_i\frac{x_i^2}{q_i},
  \qquad
  R_n(\vct{x})=\frac{n\,(\prod_ix_i)^{1/n}}{F_n(\vct{x})}.
\end{equation}
Here $F_n$ enters the covolume formula, while $R_n$ converts a bound on
$SF_n$ into a bound on the speed product.

\subsection{The ellipsoid and the covolume}

\begin{lemma}[Inclusion]\label{lem:incl}
$E\subseteq D$.
\end{lemma}

\begin{proof}
For $\vct{y}\in H$ the support functions are $h_D(\vct{y})=\sum_i|y_i|$
(the support function of a projection is the restriction of the support
function) and $h_E(\vct{y})=\sup_{|\vct{b}|\leq1}\langle
\vct{y},PQ^{1/2}\vct{b}\rangle=|Q^{1/2}\vct{y}|=(\sum_iq_iy_i^2)^{1/2}$,
since $P\vct{y}=\vct{y}$.  It suffices to show $\sum_iq_iy_i^2\leq1$ on
the polytope $\{\vct{y}\in H:\sum|y_i|\leq1\}$, and since the left side
is convex, at its vertices.  These are the points where the edges of the
cross-polytope meet $H$; an edge from $e_i$ to $-e_j$ meets $H$ at
$(v_je_i-v_ie_j)/(v_i+v_j)$, and edges between two vectors of equal sign
do not meet $H$ because $\vct{v}$ is positive.  At such a vertex,
\[
  \sum_kq_ky_k^2=\frac{q_iv_j^2+q_jv_i^2}{(v_i+v_j)^2}\leq1
  \iff (q_i-1)v_j^2+(q_j-1)v_i^2\leq2v_iv_j .
\]
With $q_i-1=u_i(2-u_i)$ and $v_i=Mu_i$, $M=\max_jv_j$, dividing by
$M^2u_iu_j$ turns the right-hand inequality into
$u_j(2-u_i)+u_i(2-u_j)\leq2$, that is, $(1-u_i)(1-u_j)\geq0$.
\end{proof}

\begin{lemma}[Covolume in the ellipsoidal norm]\label{lem:det}
The covolume of $\Lambda$ with respect to $\langle\cdot,\cdot\rangle_E$ is
\[
  \covol_E(\Lambda)=\frac{1}{S\,F_n(\vct{x})}.
\]
\end{lemma}

\begin{proof}
Let $U$ be an $n\times(n-1)$ matrix whose columns form an orthonormal
basis of $H$ and put $\tilde{\vct{v}}=\vct{v}/\lVert\vct{v}\rVert_2$.  In
the coordinates $\vct{y}=U^{\mathsf T}\vct{h}$ on $H$, the polar body
$\{\vct{h}\in H:h_E(\vct{h})\leq1\}$ has Gram matrix $U^{\mathsf T}QU$ by
the computation of $h_E$ above, so $E$ itself has Gram matrix
$(U^{\mathsf T}QU)^{-1}$ and
$\covol_E(\Lambda)=\det\Lambda\cdot\det(U^{\mathsf T}QU)^{-1/2}$.  The
orthogonal matrix $O=[U\mid\tilde{\vct{v}}]$ gives
$\det Q=\det(O^{\mathsf T}QO)
 =\det(U^{\mathsf T}QU)\cdot\bigl(\tilde{\vct{v}}^{\mathsf T}Q\tilde{\vct{v}}
 -\tilde{\vct{v}}^{\mathsf T}QU(U^{\mathsf T}QU)^{-1}U^{\mathsf T}Q\tilde{\vct{v}}\bigr)$
by the Schur complement, and the bracket is the reciprocal of the
$(n,n)$ entry of $(O^{\mathsf T}QO)^{-1}=O^{\mathsf T}Q^{-1}O$, namely
$1/(\tilde{\vct{v}}^{\mathsf T}Q^{-1}\tilde{\vct{v}})$.  Hence
$\det(U^{\mathsf T}QU)=\det Q\cdot\tilde{\vct{v}}^{\mathsf T}Q^{-1}\tilde{\vct{v}}
=\prod_iq_i\cdot\sum_iv_i^2/q_i\big/\lVert\vct{v}\rVert_2^2$, and
\[
  \covol_E(\Lambda)^2
  =\frac{1}{\lVert\vct{v}\rVert_2^2}\cdot
   \frac{\lVert\vct{v}\rVert_2^2}{\prod_iq_i\sum_iv_i^2/q_i}
  =\frac{1}{S^2\prod_iq_i\sum_ix_i^2/q_i}
  =\frac{1}{S^2F_n(\vct{x})^2}.\qedhere
\]
\end{proof}

\subsection{Points with all coordinates far from integers}

The following lemma is Lemma~3.3 of Giri and Kravitz~\cite{GK26} in a
different normalization: in their notation it says
$\max\mathcal S_q(n)=\max\mathcal S_1(n+1-q)$, and $LRC(n+1-q)$ bounds
the right-hand side.  We include a short proof along their lines, since
the statement is used with every $q$ below.

\begin{lemma}[Separated points on rational subspaces]\label{lem:safe}
Let $1\leq q\leq n$ and assume $LRC(m)$ for every $m\leq n+1-q$.  Let
$V\subseteq\R^n$ be a rational subspace of dimension $q$ on which no
coordinate functional vanishes identically.  Then $V$ contains a point
$\vct{s}$ with
\[
  \nearest{s_i}\geq\frac{1}{n+2-q}\qquad(1\leq i\leq n).
\]
\end{lemma}

\begin{proof}
Induction on $q$, simultaneously for all $n$.  For $q=1$, $V=\R\vct{a}$
with $\vct{a}\in\Z^n$ primitive and all $a_i\neq0$.  The values $|a_i|$
are $m\leq n$ distinct positive integers;
$LRC(m)$ gives $t$ with $\nearest{ta_i}\geq1/(m+1)\geq1/(n+1)$, and
$\vct{s}=t\vct{a}$ works.

For $q\geq2$, we will pass to a hyperplane on which two coordinates agree
up to sign and no coordinate vanishes.  We can then delete one coordinate
and apply induction in $\R^{n-1}$.
The restrictions of the coordinate functionals to $V$ are
nonzero vectors $\vct{a}_1,\ldots,\vct{a}_n\in V$ (identify $V^{*}$ with
$V$ through the Euclidean inner product); they span $V^{*}$, so not all of
them are proportional.  Among all non-proportional pairs choose $(i,j)$
minimizing the angle between the lines $\R\vct{a}_i$ and $\R\vct{a}_j$,
and replace $\vct{a}_j$ by $-\vct{a}_j$ if necessary so that
$\langle\vct{a}_i,\vct{a}_j\rangle\geq0$; this amounts to changing the
sign of the $j$-th coordinate, which does not affect $\nearest{\cdot}$.
Put $\vct{h}=\vct{a}_i+\vct{a}_j\neq0$.  The line $\R\vct{h}$ makes a
strictly smaller angle with $\R\vct{a}_i$ than $\R\vct{a}_j$ does, so by
minimality $\vct{h}$ is proportional to no $\vct{a}_k$.  Let
$V'=\{\vct{y}\in V:\langle\vct{h},\vct{y}\rangle=0\}$, a rational
subspace of dimension $q-1$ on which $y_i=-y_j$ (after the sign change).
No coordinate vanishes on $V'$: the functional $\vct{a}_k$ vanishes on
$V'=\vct{h}^{\perp}\cap V$ only if $\vct{a}_k\in\R\vct{h}$.  Deleting the
$j$-th coordinate maps $V'$ injectively onto a rational subspace
$V''\subseteq\R^{n-1}$ of dimension $q-1$ with no vanishing coordinate.
  The induction hypothesis applies to $V''$ with $n-1$ in place of $n$; it
  requires $LRC(m)$ for $m\leq(n-1)+1-(q-1)=n+1-q$.  It gives
  $\vct{s}''\in V''$ with all coordinates at distance at least
  $1/((n-1)+2-(q-1))=1/(n+2-q)$.  Its preimage $\vct{s}\in V'$ has
  $\nearest{s_j}=\nearest{-s_i}=\nearest{s_i}$, so all $n$ coordinates
  satisfy the bound.
\end{proof}

\subsection{Prefix inequalities for the Gram--Schmidt lengths}

Let $\vct{b}_1,\ldots,\vct{b}_d$ be any basis of $\Lambda$ ($d=n-1$), let
$\vct{b}_i^{*}$ be its Gram--Schmidt vectors with respect to
$\langle\cdot,\cdot\rangle_E$, and put $t_i=\Enorm{\vct{b}_i^{*}}^2>0$.
Then $\prod_it_i=\covol_E(\Lambda)^2$.

If the first $r$ Gram--Schmidt lengths were too small, nearest-plane
rounding would turn the point from Lemma~\ref{lem:safe} into a witness for
the speed tuple.  This bounds each initial segment from below.

\begin{proposition}[Prefix inequalities]\label{prop:prefix}
Assume $LRC(m)$ for all $m<n$ and let $\vct{v}$ be a counterexample.  Then
for every $r=1,\ldots,d$,
\begin{equation}\label{eq:prefix}
  t_1+\cdots+t_r\;>\;B_r,\qquad
  B_r:=\Bigl(\frac{2r}{(n+1)(n+1-r)}\Bigr)^{2}.
\end{equation}
\end{proposition}

\begin{proof}
Fix $r$ and choose $\vct{z}_i\in\Z^n$ with $P\vct{z}_i=\vct{b}_i$
($1\leq i\leq r$).  The subspace $V=\spn(\vct{v},\vct{z}_1,\ldots,\vct{z}_r)$
is rational and has dimension $r+1$: if
$a\vct{v}+\sum a_i\vct{z}_i=0$, applying $P$ gives $\sum a_i\vct{b}_i=0$,
so all $a_i=0$ and then $a=0$.  No coordinate vanishes on $V$, since
$\vct{v}\in V$ has positive coordinates.  Lemma~\ref{lem:safe} with
$q=r+1$ (it needs $LRC(m)$ for $m\leq n-r\leq n-1$) gives
$\vct{s}=a\vct{v}+\sum_ia_i\vct{z}_i\in V$ with
$\nearest{s_k}\geq1/(n+1-r)$ for all $k$.

Apply nearest-plane rounding to the coefficients: choose integers
$k_r,k_{r-1},\ldots,k_1$ in this order so that
$\vct{e}:=\sum_{i\leq r}(a_i-k_i)\vct{b}_i=\sum_{i\leq r}\theta_i\vct{b}_i^{*}$
has $|\theta_i|\leq\frac12$ for all $i$.  At step $i$, the coefficient of
$\vct{b}_i^{*}$ is $a_i-k_i$ plus a quantity already fixed by
$k_{i+1},\ldots,k_r$.  We can therefore choose $k_i$ to bring it into
$[-\frac12,\frac12]$.  By orthogonality,
\begin{equation}\label{eq:round}
  \varepsilon:=\Enorm{\vct{e}}\leq\tfrac12\sqrt{t_1+\cdots+t_r}.
\end{equation}
By Lemma~\ref{lem:incl}, $\vct{e}\in\varepsilon E\subseteq\varepsilon D$,
so $\vct{e}=P\vct{y}$ for some $\vct{y}\in\R^n$ with $|y_k|\leq\varepsilon$
for all $k$.  The vector $\sum_i(a_i-k_i)\vct{z}_i-\vct{y}$ has projection
$\vct{e}-\vct{e}=0$, so it equals $c\vct{v}$ for some real $c$.  It follows that
\[
  \vct{s}-\sum_ik_i\vct{z}_i-\vct{y}=(a+c)\vct{v}=:t\vct{v},
\]
and since $\sum_ik_i\vct{z}_i\in\Z^n$,
\[
  \nearest{tv_k}=\nearest{s_k-y_k}\geq\nearest{s_k}-|y_k|
  \geq\frac{1}{n+1-r}-\varepsilon\qquad(1\leq k\leq n).
\]
If $\varepsilon\leq\frac1{n+1-r}-\frac1{n+1}=\frac{r}{(n+1)(n+1-r)}$,
then $t$ is a witness for $\vct{v}$, contradicting that $\vct{v}$ is a
counterexample.  So $\varepsilon>r/((n+1)(n+1-r))$, and
\eqref{eq:round} gives $t_1+\cdots+t_r\geq4\varepsilon^2>B_r$.
\end{proof}

The rounding argument has a geometric interpretation.  Every point of
the rational subtorus spanned by $\vct{v},\vct{z}_1,\ldots,\vct{z}_r$
lies within $\tfrac12\sqrt{t_1+\cdots+t_r}$ of the orbit
$t\mapsto t\vct{v}\pmod{\Z^n}$ in each coordinate.  If this distance
bound were at most $r/((n+1)(n+1-r))$, the point supplied by
Lemma~\ref{lem:safe} would give a witness, as the proof shows.
For $r=1$ this recovers the two-dimensional subtorus argument of
Giri and Kravitz~\cite[\S7]{GK26}, with a sharper estimate.  For $r=d$
the inequality reads $\sum_it_i>((n-1)/(n+1))^2$ and concerns the whole
lattice.  We use all these inequalities together, one for each initial
segment of the basis.

\subsection{Korkine--Zolotarev bases and the product minimization}

\begin{lemma}[KZ bases]\label{lem:kz}
$\Lambda$ has a basis $\vct{b}_1,\ldots,\vct{b}_d$ whose Gram--Schmidt
lengths satisfy $t_{i+1}\geq\frac34t_i$ for $1\leq i<d$.
\end{lemma}

\begin{proof}
A Korkine--Zolotarev basis exists for every lattice~\cite{KZ1873,LLS90}.
In the projected lattice, $\vct{b}_i^*$ is shortest, while size reduction
of $\vct{b}_{i+1}$ gives a nonzero vector of squared length at most
$t_{i+1}+t_i/4$.  Hence $t_i\leq t_{i+1}+t_i/4$.
\end{proof}

Put $\rho=\frac34$, $B_0=0$, and $c_i=B_i-B_{i-1}$ for $1\leq i\leq d$.
We now bound $\prod_i t_i$, the squared covolume.  Under the hypotheses
below, its minimum occurs when every partial-sum constraint is an equality.

\begin{lemma}[Product minimization]\label{lem:minimize}
Suppose $c_{i+1}/c_i>1/\rho$ for $1\leq i<d$.  If positive reals
$t_1,\ldots,t_d$ satisfy $t_{i+1}\geq\rho t_i$ for all $i<d$ and
$t_1+\cdots+t_r\geq B_r$ for all $r\leq d$, then
$\prod_it_i\geq\prod_ic_i$, with equality at $t_i=c_i$.
\end{lemma}

\begin{proof}
We reduce to vertices, then compare products block by block.
Put $w_i=\rho^{i-1}$ and $z_i=t_i/w_i$.  The constraints become
$z_1\leq z_2\leq\cdots\leq z_d$ and $\sum_{i\leq r}w_iz_i\geq B_r$
($1\leq r\leq d$); they define a polyhedron $\Pi\subset\R^d$ contained in
the open positive orthant ($z_1\geq B_1>0$).  The objective
$f(\vct{z})=\sum_i\log(w_iz_i)$ is concave on $\Pi$ and nondecreasing in
every coordinate, and the recession cone of $\Pi$,
$\{0\leq z_1\leq\cdots\leq z_d\}$, lies in the nonnegative orthant.
$\Pi$ contains no line, so every point of $\Pi$ is a convex combination of
vertices plus a recession vector, and $f$ is at least its minimum over the
vertices.  It remains to show $f\geq\sum\log c_i$ at every vertex.

At a vertex, $d$ linearly independent constraints are active.  The active
monotonicity constraints $z_i=z_{i+1}$ partition $\{1,\ldots,d\}$ into
maximal blocks of constant $z$; if there are $\beta$ blocks, then
$d-\beta$ monotonicity constraints are active and at least $\beta$ prefix
constraints must be active.  No prefix constraint is active at an index
$r$ strictly inside a block $\{a+1,\ldots,b\}$ ($a<r<b$) on which $z$
equals $\zeta$.  Suppose otherwise.  Feasibility at $a$ and activity at
$r$ give
$\zeta\sum_{a<i\leq r}w_i\leq B_r-B_a=\sum_{a<i\leq r}c_i$.
Thus $\zeta$ is at most a weighted average of the $c_i/w_i$ over
$a<i\leq r$.  Since these ratios are strictly increasing,
$\zeta\leq c_r/w_r$.  On the other hand, activity at $r$ and feasibility
at $r+1$ give $w_{r+1}\zeta\geq c_{r+1}$, and hence
$\zeta\geq c_{r+1}/w_{r+1}>c_r/w_r$, a contradiction.  Thus the
active prefixes can occur only at the $\beta$ block ends.  Since at least
$\beta$ are needed, every block end is active.  On each block $\{a+1,\ldots,b\}$ the
$t_i=w_i\zeta$ form a geometric sequence with ratio $\rho$ whose sum is
$B_b-B_a=\sum_{a<i\leq b}c_i$.

\Needspace{5\baselineskip}
Compare, within one block of length $L$, the decreasing geometric
sequence $\vct{g}=(t_{a+1},\ldots,t_b)$ with the sequence
$\vct{c}^{\downarrow}=(c_b,c_{b-1},\ldots,c_{a+1})$, both positive with
the same sum.  Consecutive ratios satisfy
$c^{\downarrow}_{j+1}/c^{\downarrow}_j<\rho=g_{j+1}/g_j$, so for $j\leq
k<l$ we have $c^{\downarrow}_l/c^{\downarrow}_j\leq g_l/g_j$ and
$c^{\downarrow}_j/c^{\downarrow}_k\geq g_j/g_k$, whence
\[
  \frac{\sum_{l>k}c^{\downarrow}_l}{\sum_{j\leq k}c^{\downarrow}_j}
  \leq\frac{\sum_{l>k}g_l/g_k}{\sum_{j\leq k}g_j/g_k}
  =\frac{\sum_{l>k}g_l}{\sum_{j\leq k}g_j}
  \qquad(1\leq k<L),
\]
and with equal totals, $\sum_{j\leq k}c^{\downarrow}_j\geq\sum_{j\leq
k}g_j$ for all $k$: $\vct{c}^{\downarrow}$ majorizes $\vct{g}$.  Since
$\sum\log$ is Schur-concave, $\prod_{a<i\leq b}c_i\leq\prod_{a<i\leq
b}t_i$.  Multiplying over the blocks gives $\prod_it_i\geq\prod_ic_i$ at
every vertex.  Finally $t_i=c_i$ is feasible ($c_{i+1}>c_i/\rho>\rho
c_i$ and all prefixes are equalities), so the bound is attained.
\end{proof}

\subsection{The shape factor}

\begin{lemma}[Shape bound]\label{lem:shape}
On the positive simplex,
$R_{13}<101/200$ and $R_{14}<1/2$.
\end{lemma}

\begin{proof}
We start with $n=14$.  At a minimizer, all coordinates below the maximum
have the same value, reducing the estimate to one variable.  We prove this
first, then give the changes for $n=13$.
Passing from $x_i$ to $u_i=x_i/\max_jx_j$ cancels the common scale
in \eqref{eq:F}: with
\[
  q(u)=1+2u-u^2,\qquad a(u)=\frac{q(u)}{u^{2/n}},\qquad
  w(u)=\frac{u^2}{q(u)},\qquad
  H(\vct{u})=\Bigl(\sum_iw(u_i)\Bigr)\prod_ia(u_i),
\]
one has $R_n(\vct{x})^2=n^2/H(\vct{u})$, so it suffices to prove
$H>4n^2=784$ on the domain $\vct{u}\in(0,1]^n$ with some $u_i=1$.  Since
$w(1)=\frac12$, $w\geq0$, and $a(u)\to\infty$ as $u\to0$, $H$ tends to
infinity when any coordinate tends to $0$, so $H$ attains its minimum.

\emph{Interior coordinates.}  Fix all coordinates but one interior
coordinate $u\in(0,1)$ and put $C=\sum_{j\neq i}w(u_j)\in(0,\frac{n-1}2]$.
Up to a positive factor the objective is $\phi(u)=a(u)(C+w(u))$, and a
direct computation gives
\begin{equation}\label{eq:phi}
  \tfrac n2\,u^{1+2/n}\,\phi'(u)
  =-\bigl[C\{(n-1)u^2-(n-2)u+1\}-(n-1)u^2\bigr]=:-\beta(u).
\end{equation}
Let $\alpha$ be the smaller root of $(n-1)u^2-(n-2)u+1$ ($\alpha_{14}
=0.0926\ldots$).  Then $\beta(0)=C>0$ and $\beta(\alpha)=-(n-1)\alpha^2<0$.
If $C\leq1$, $\beta$ is concave with $\beta(0)>0$ and has exactly one
positive root; if $C>1$, $\beta$ is convex with $\beta(1)=2C-(n-1)\leq0$,
so its second root is $\geq1$.  In both cases $\beta$ has exactly one
root $u_0\in(0,\alpha)$ in $(0,1)$, and $\phi$ decreases on $(0,u_0)$ and
increases on $(u_0,1)$.  Hence at a minimizer every interior coordinate is
a critical point in $(0,\alpha)$.  At such a point, with
$A=C+w(u)=\sum_iw(u_i)$, the equation $\beta(u)=0$ is equivalent to
\begin{equation}\label{eq:crit}
  \frac{q(u)\{1-(n-2)u+(n-1)u^2\}}{n\,u^2(1+u)}=\frac1A .
\end{equation}
The derivative of the left side has the sign of
$-(13u^5+26u^4-50u^3-20u^2-7u+2)$, and the polynomial is positive on
$(0,\alpha)$ (drop its positive terms and use $\alpha<0.1$), so the left
side is strictly decreasing there.  Since $A$ is the same for every
coordinate, all interior coordinates of a minimizer share one value $t$.

\emph{Two or more coordinates equal to one.}  We use $a(u)>\frac85$ on
$(0,1]$: the logarithmic derivative of $a$ vanishes exactly at the roots
$(6\pm\sqrt{23})/13$ of $13u^2-12u+1$, so $a$ decreases from $+\infty$
to a local minimum at $u_1=(6-\sqrt{23})/13\in(0.092,0.093)$, increases
to a local maximum, and decreases to $a(1)=2$; hence
$a\geq\min(a(u_1),2)$ and $a(u_1)>q(0.092)/0.093^{1/7}>1.65$.  If $k\geq2$
coordinates equal one, then $\sum_iw(u_i)\geq k/2$ and
$H\geq\frac k2\,2^k(\tfrac85)^{n-k}$, which increases with $k$ and at
$k=2$ equals $4\cdot(8/5)^{12}>1125>784$.

\emph{Exactly one coordinate equal to one.}  Then the other thirteen equal
$t\in(0,\alpha)$ and
\[
  H(t)=\Bigl(\tfrac12+\tfrac{13t^2}{q(t)}\Bigr)\cdot2\cdot\frac{q(t)^{13}}{t^{26/14}}
  =\frac{(1+2t+25t^2)\,q(t)^{12}}{t^{13/7}} .
\]
On this interval its logarithmic derivative has the same sign as
$325t^3+23t^2+9t-1$, since all remaining factors are positive.  Thus $H$
decreases until the unique root
$t_0\in(0.0782,0.0783)$ of the increasing cubic and increases afterwards.
As $q$ and $1+2t+25t^2$ increase and $t^{-13/7}$ decreases,
\[
  H(t_0)\geq\frac{(1+2\cdot0.0782+25\cdot0.0782^2)\,q(0.0782)^{12}}{0.0783^{13/7}}
  >796.4>784 ,
\]
an exact rational comparison after raising to the seventh power.  So
$H>784$ in every case, and $R_{14}<14/28=\frac12$.

For $n=13$, the smaller root in the interior-coordinate argument is
$\alpha_{13}=(11-\sqrt{73})/24<0.103$.  In~\eqref{eq:crit} the derivative
has the sign of $-2P(u)$, where
$P(u)=6u^5+12u^4-23u^3-9u^2-3u+1$.  On $(0,\alpha_{13})$,
$P(u)>1-3(0.103)-9(0.103)^2-23(0.103)^3>0$, so again all interior
coordinates share one value.  With two or more coordinates equal to one,
the corresponding one-variable check gives
$a(u)>8/5$, and the preceding product estimate gives
$H\geq4(8/5)^{11}>(2600/101)^2$.  With exactly one such coordinate, the
other twelve share the unique value $t_0\in(0.0849,0.0850)$, since the
logarithmic derivative of $H$ has the sign of the increasing cubic
$276t^3+21t^2+8t-1$, and
\[
 H(t_0)=\frac{(1+2t_0+23t_0^2)q(t_0)^{11}}{t_0^{24/13}}
 >663.48>\left(\frac{2600}{101}\right)^2.
\]
The last inequality follows by using $0.0849$ in the increasing factors
and $0.0850$ in the denominator, and may be checked after raising to the
thirteenth power.  This is exactly $R_{13}<101/200$.
\end{proof}

\subsection{The theorem and its constants}

\begin{theorem}[Flag bound]\label{thm:flag}
Assume $LRC(m)$ for all $m<n$, let $\vct{v}\in\Z^n_{>0}$ be a primitive
counterexample to $LRC(n)$, and suppose $c_{i+1}/c_i>4/3$ for
$1\leq i<n-1$.  Put
\[
  K_n=\prod_{i=1}^{n-1}c_i,\qquad A_n=K_n^{-1/2}.
\]
Then $S\,F_n(\vct{x})\leq A_n$.  If moreover $R_n<r_n$ on the
simplex, then
\[
  v_1\cdots v_n<\Bigl(\frac{A_n r_n}{n}\Bigr)^{n}.
\]
For $n=13$, Lemma~\ref{lem:shape} gives
\[
  A_{13}=6362078580113.54\ldots,
  \qquad
  \log(v_1\cdots v_{13})<341.031991.
\]
For $n=14$, the corresponding constants are
\[
  K_{14}=\frac{2289670297}{97243124257250844122250000000000000000},\qquad
  A_{14}=206083383792625.44\ldots,
\]
\[
  \log(v_1\cdots v_{14})<14\log\frac{A_{14}}{28}
  =414.77936455669981\ldots<\RawThreshUB .
\]
\end{theorem}

\begin{proof}
Take a KZ basis (Lemma~\ref{lem:kz}).  Its Gram--Schmidt lengths satisfy
the hypotheses of Lemma~\ref{lem:minimize} by Proposition~\ref{prop:prefix},
so $\covol_E(\Lambda)^2=\prod_it_i\geq K_n$, and Lemma~\ref{lem:det} gives
$(SF_n(\vct{x}))^{-2}\geq K_n$, i.e.\ $SF_n(\vct{x})\leq A_n$.  By the
definition \eqref{eq:F} of $R_n$,
\[
  n\,(v_1\cdots v_n)^{1/n}=n\,S\,(\textstyle\prod_ix_i)^{1/n}
  =S\,F_n(\vct{x})\,R_n(\vct{x})<r_nA_n ,
\]
which is the product bound.  The ratio hypothesis has minimum
$1175/688$ for $n=13$ and $847/513$ for $n=14$; Lemma~\ref{lem:shape}
supplies $r_{13}=101/200$ and $r_{14}=1/2$.  Substitution gives the
displayed constants.
\end{proof}

\begin{remark}[Comparison with the earlier bound]\label{rem:compare}
The bound of~\cite{MSS25}, in the form $\sum_iv_i<\binom{n+1}{2}^{n-1}$
used in~\cite{ST26}, gives
$\log(v_1\cdots v_{14})<14(13\log105-\log14)=810.07\ldots$.
The recent reformulation in~\cite{BCS26} gives the same bound; the bound
of~\cite{GK26} is weaker.

The prefix inequalities already give a substantial improvement.
Let $\lambda_1\leq\cdots\leq\lambda_d$ be the successive minima of
$\Lambda$ with respect to $D$.  Rounding the coefficients of independent
vectors attaining these minima, as in the proof of
Proposition~\ref{prop:prefix}, gives
$\lambda_1+\cdots+\lambda_r>2r/((n+1)(n+1-r))$ for every $r$.
The projected cube $D$ has volume
$2^{n-1}S/\lVert\vct{v}\rVert_2$.  Together with Minkowski's second
theorem, these inequalities yield $\sum_iv_i<(n+1)!\,n!/2^n$,
giving a logarithmic product bound of $570.46$ when $n=14$
(Table~\ref{tab:thresholds}).

Two further ingredients bring this down to $414.78$.  With
Gram--Schmidt lengths we have the exact identity
$\prod_it_i=\covol_E(\Lambda)^2$.  The KZ basis adds the ratio constraints
needed to minimize this product in Lemma~\ref{lem:minimize}, so we no
longer need an estimate from Minkowski's second theorem.
The adapted ellipsoid $E$ lets us pass directly from the covolume
estimate to the speed product, using the shape factor of
Lemma~\ref{lem:shape}, without first converting a Euclidean bound into
one on the sum of the speeds.  For $n=13$ the theorem gives
$341.03\ldots$ against $670.35\ldots$.  We are not aware of an earlier
improvement of this bound.
\end{remark}

\subsection{The forced divisor}

We also use the following divisibility observation of
Rosenfeld~\cite[Lemma~4]{Rosenfeld25}.

\begin{lemma}[Forced divisor]\label{lem:divisor}
If $\vct{v}\in\Z_{>0}^{n}$ is a counterexample to $LRC(n)$, then the
integer $\lcm(2,3,\ldots,n+1)$ divides $v_1\cdots v_n$.  For $n=14$,
\[
  \lcm(2,\ldots,15)=360360=2^3\cdot3^2\cdot5\cdot7\cdot11\cdot13,\qquad
  \log360360=12.79485\ldots>\DivisorLogLB .
\]
\end{lemma}

\begin{proof}
Fix $2\leq d\leq n+1$.  If no $v_i$ is divisible by $d$, then at
$t=1/d$ every $\nearest{v_i/d}$ is a nonzero multiple of $1/d$, hence
$\geq1/d\geq1/(n+1)$, and $t$ is a witness.  So every $d\leq n+1$
divides some $v_i$.  For each prime $r\leq n+1$, let $q_r$ be the largest
power of $r$ not exceeding $n+1$.  Then $q_r$ divides the speed product.
The numbers $q_r$ are pairwise coprime, and their product is
$\lcm(2,\ldots,n+1)$.
\end{proof}

\begin{corollary}[Sufficient logarithmic prime sum]\label{cor:mass}
Let $\mathcal P$ be a set of primes $p>13$ such that, for every $p\in
\mathcal P$, every primitive counterexample
$\vct{v}\in\Z_{>0}^{14}$ to $LRC(14)$ with distinct coordinates has
$p\mid v_1\cdots v_{14}$.  If
$\sum_{p\in\mathcal P}\log p>\TargetUB$, then $LRC(14)$ holds.
\end{corollary}

\begin{proof}
We use $LRC(13)$, proved first in Section~\ref{sec:gateset}.
Let $\vct{v}$ be a counterexample.  Sign changes and rescaling preserve the
lonely-runner property, so we may take $\vct{v}$ positive and primitive.
If two coordinates coincide, at most thirteen distinct speeds occur, and
$LRC(13)$ gives a witness with
$\nearest{tv_i}\geq\frac1{14}>\frac1{15}$.  The coordinates must therefore
be pairwise distinct.  Then $360360\prod_{p\in\mathcal
P}p$ divides $v_1\cdots v_{14}$, the factors being pairwise coprime, so
\[
  \log(v_1\cdots v_{14})\geq\log360360+\sum_{p\in\mathcal P}\log p
  >\DivisorLogLB+\TargetUB=\RawThreshUB>14\log\frac{A_{14}}{28},
\]
contradicting Theorem~\ref{thm:flag}.
\end{proof}

\section{The modular computations}\label{sec:pipeline}

Fix $k\in\{13,14\}$ and a prime $p$.  After folding signs, both speed
classes and nonzero time classes are represented by
$1,\ldots,(p-1)/2$.  Put
\[
 d_p(x)=\min(r,p-r),\qquad r\equiv x\pmod p,\quad 0\leq r<p.
\]
A speed class $v$ covers a time class $a$ when
\begin{equation}\label{eq:cover}
  (k+1)d_p(av)<p.
\end{equation}
Note that this strict inequality is precisely the failure of the witness
inequality for the speed $v$ at time $a/p$.
At level one the gcd alternative in the definition of properness is
unavailable.  Thus $I(k,p,1)$ is exactly the family of $k$-multisets whose
speed classes cover every folded time class.

Let $\tau_k(p)$ be the smallest number of speed classes that cover all
folded time classes under~\eqref{eq:cover}.  When $k$ is fixed, we write
$\tau(p)$.

\subsection{A complete two-branch covering search}\label{sec:covers}

Every such multiset is of one of two kinds.  It is an irredundant cover on
$k$ distinct classes, or it contains a cover on at most $k-1$ classes.
We search these two cases separately.  At each node we choose an uncovered
time that can be covered by the fewest available speed classes.  Every
completion must use one of those classes, so we branch on these choices.
We discard a branch if the remaining classes, even at best, cannot cover
all uncovered times.  The next lemma gives a further pruning rule.

\begin{lemma}[Private-time bound]\label{lem:quota}
Let $n=(p-1)/2$ and $m=\lfloor(p-1)/(k+1)\rfloor$.  In a cover of the $n$
folded times by at most $s$ distinct speed classes, some class covers at
least
\[
 q_s=\left\lceil\frac{\max(0,2n-sm)}{s}\right\rceil
\]
times privately, meaning that no other class in the cover covers them.  If
the cover is the support of a $k$-multiset on at most $k-1$ classes, such a
class can be chosen with multiplicity at most two and with at least
$q_{k-1}$ private times.
\end{lemma}

\begin{proof}
Suppose that $s'\leq s$ classes form the cover, and let $n_j$ be the number
of times covered exactly $j$ times.  Counting incidences between speed
classes and time classes gives
\[
 s'm=\sum_jj n_j\geq n_1+2(n-n_1),
\]
so $n_1\geq2n-s'm\geq2n-sm$.  These $n_1$ times are private, and one
class has at least $q_s$ private times.

For the refinement, let $h$ be the number of support classes of
multiplicity at least three.  Since the multiset has $k$ entries,
$s'+2h\leq k$, hence $s'+h\leq k-1$.  The classes of multiplicity at least
three account for at most $hm$ private times.  The other classes have at
least
\[
 2n-s'm-hm\geq2n-(k-1)m>0
\]
of them.  There are at most $k-1$ such classes, so one has multiplicity at
most two and at least $q_{k-1}$ private times.
\end{proof}

The irredundant branch enumerates $k$-class covers that satisfy the $q_k$
bound.  In the smaller-cover branch, take a $k$-multiset $M$ supported on at most $k-1$
classes and choose the class $c$ given by Lemma~\ref{lem:quota}.  If $c$
occurs twice, remove one copy; if it occurs once, remove a copy of some
repeated class.  The resulting $(k-1)$-multiset has the same support and
contains $c$ exactly once.  A unit moves $c$ to class $1$.  The search
enumerates the resulting multiset satisfying the $q_{k-1}$ bound and every
one-class extension, one of which recovers $M$.  The same argument applies when $M$
has $k$ distinct classes but one is removable: remove that class first.

\begin{proposition}[Completeness of level-one generation]
\label{prop:basecomplete}
For $k=13$ and $k=14$, the corresponding procedure produces at least one
representative of every unit orbit in $I(k,p,1)$.
\end{proposition}

\begin{proof}
Equation~\eqref{eq:cover} identifies the base family with the covers above.
The two branches exhaust those covers, and Lemma~\ref{lem:quota} proves that
each orbit has a normalization satisfying the private-time bound.
Branching on an uncovered time includes every completion; the gain bound
discards only branches that cannot complete a cover.  Taking a
lexicographic minimum in the irredundant branch removes duplicates without
removing an orbit; the smaller-cover branch may retain more than one
representative.
\end{proof}

Before running the second branch, a separate exhaustive search can certify
a lower bound for $\tau(p)$.  This sometimes shortens the calculation.  For
$k=13$, if $\tau(p)\geq12$, it is enough to enumerate twelve-class covers;
otherwise we enumerate all covers on at most twelve classes.  For $k=14$,
if $\tau(p)\geq13$, we extend only minimal thirteen-class covers; if
$\tau(p)\leq12$, we use the general enumeration above.  These shortcuts
change only the amount of work, not the family covered by
Proposition~\ref{prop:basecomplete}.

\subsection{Successive binary lifting}

For each retained level-one representative $r$, we use the binary lifts
of Section~\ref{sec:framework} to compute $F_l(r)$ at levels
$2,4,8,16,$ and $32$, stopping if the set becomes empty.

For each of the $111$ primes in the
\href{https://doi.org/10.5281/zenodo.22066772}{fourteen-runner archive},
every level-one orbit except two has an empty improper fiber by level
$32$.  The two persistent orbits are represented by
\begin{equation}\label{eq:orbits13}
 \vct{a}_{13}=(1,2,\ldots,13),\qquad
 \vct{b}_{13}=(1,2,\ldots,11,13,24).
\end{equation}
We call a $k$-tuple \emph{tight} if
$\max_t\min_i\nearest{tv_i}=1/(k+1)$.  Both tuples in~\eqref{eq:orbits13}
are tight; the second is the Goddyn--Wong example~\cite{GW06}.
For the verified $k=14$ primes with $\tau(p)\geq13$, the only orbit surviving the binary lifts
is represented by $\vct{a}_{14}=(1,2,\ldots,14)$.  When
$\tau(p)\leq12$, more orbits survive and are handled by the level-$15$
argument below.

\subsection{Fourteen runners: the \texorpdfstring{level-$14$}{level-14} fibers}

Each improper level-$2$ member of the two orbits in
\eqref{eq:orbits13} has $7^{13}$ lifts to level $14$.  We assign the lift
digits one at a time.  At a partial assignment, let $U$ be the uncovered
time classes, and for each unassigned coordinate $i$ let $b_i$ be the
largest number of classes in $U$ covered by any of its seven choices.  If
$|U|>\sum_i b_i$, no completion can cover all remaining time classes, so we
discard the branch.  Every other branch is continued to a leaf and checked
there.

\begin{lemma}[Completeness of the level-$14$ search]\label{lem:terminal14}
If every witness-free tuple completed by the preceding search satisfies
the gcd condition, then the fiber contains no improper tuple.
\end{lemma}

\begin{proof}
The unassigned coordinates cover at most $\sum_i b_i$ classes of $U$ in
any completion.  A discarded branch therefore leaves an uncovered time and
every completion in it has a witness.  All remaining leaves are tested
directly against both alternatives in the definition of properness.
\end{proof}

At every prime in this archived $111$-prime set, every
completed tuple without a witness had all thirteen coordinates divisible by
$7$.  Omitting any coordinate leaves a common divisor with the level, so each
such tuple is proper by the gcd condition.  Lemma~\ref{lem:terminal14}
shows that every level-$14$ lift of either persistent orbit is proper.

\begin{proposition}\label{prop:prime13}
For every prime $p$ in the archived $111$-prime set,
$J(13,p)=\varnothing$.  In particular, this holds for
every $p\in\mathcal P_{13}$ defined in Section~\ref{sec:gateset}.
\end{proposition}

\begin{proof}
By Proposition~\ref{prop:basecomplete}, the initial search includes a
representative of every level-one orbit.  The binary lifts show eventual
properness for all but the two orbits in~\eqref{eq:orbits13}.  For these two,
the level-$14$ computation and Lemma~\ref{lem:terminal14} show that every
lift is proper.  Thus every orbit is eventually proper, and
$J(13,p)=\varnothing$.
\end{proof}

\subsection{Fifteen runners: factoring \texorpdfstring{level $15$}{level 15}}

A direct search of all $15^{14}$ lifts in a final fiber would be too large.
Instead we use $15=3\cdot5$.

\begin{lemma}[Witnesses lift along divisors]\label{lem:factor}
Let $f\mid15$, let $\vct{w}$ be a lift of $r$ to level $15$, and let
$\vct{w}'$ be its reduction to level $f$.  A witness for $\vct{w}'$ in
$(fp)^{-1}\Z$ is also a witness for $\vct{w}$.  Consequently every
witness-free level-$15$ lift reduces to witness-free lifts at levels $3$
and $5$, which determine it through the Chinese remainder theorem.
\end{lemma}

\begin{proof}
If $\vct{w}\equiv\vct{w}'\pmod{fp}$ and $t\in(fp)^{-1}\Z$, then
$t(w_i-w_i')\in\Z$, so
$\nearest{tw_i}=\nearest{tw_i'}$.  The lift digits modulo $3$ and modulo
$5$ determine their residues modulo $15$.
\end{proof}

We search the $3^{14}$ lifts at level $3$ and the $5^{14}$ lifts at level
$5$ separately, then combine compatible candidates using the Chinese
remainder theorem.  We test each resulting tuple on the full level-$15$
time grid and check the gcd condition.  For $\vct{a}_{14}$, seven candidates
have no grid witness, and all seven satisfy the gcd condition.

When the covering number is smaller, some lifts fail the gcd test despite
having many coordinates divisible by $3$ or $5$.  Shifts by $1/d$ leave the
speeds divisible by $d$ fixed modulo one, allowing us to handle the few
exceptions.  The next lemma adapts
Rosenfeld's shifting argument~\cite[Lemma~5]{Rosenfeld26} to fourteen speeds.

\begin{lemma}[Few-exception shift]\label{lem:neargcd}
Let $u_1,\ldots,u_{14}$ be positive integers, let $d$ be prime, let
$e=\#\{i:d\nmid u_i\}$, and put $m_d=\lceil2d/15\rceil$.  Assume
$LRC(m)$ for $m\leq12$.  If $2\leq e$ and $e m_d<d$, then the tuple has
the LR property.  In particular, it applies when exactly twelve
coordinates are divisible by $3$, or when ten, eleven, or twelve
coordinates are divisible by $5$.
\end{lemma}

\begin{proof}
The speeds divisible by $d$ take at most $14-e\leq12$ distinct values, so
$LRC(14-e)$ gives a time $t_0$ at which they are all at distance at least
$1/(15-e)>1/15$ from the integers.  The shifts
$t_0+q/d$, $0\leq q<d$, leave these coordinates fixed modulo one.  For an
exceptional speed, the shifts form a complete $d$-point grid, and the bad
open arc of length $2/15$ contains at most $m_d$ of its points.  All $e$
exceptions forbid fewer than $d$ shifts, so one shift is good for every
coordinate.  For $d=3$ we have $m_3=1$ and $e=2$; for $d=5$ we have
$m_5=1$ and $2\leq e\leq4$.
\end{proof}

For either $d\in\{3,5\}$, the cases $e\leq1$ are already proper by the
gcd condition.  The level-$15$ program applies Lemma~\ref{lem:neargcd} to
every witness-free lift left improper by that test.  Divisibility by $3$
and $5$ is constant on its residue class modulo $15p$, so the third
alternative of Lemma~\ref{lem:gate}(ii) holds for that entire class.

\begin{proposition}\label{prop:prime14}
For every $p\in\mathcal P_{14}$ defined below, every primitive
counterexample to $LRC(14)$ with distinct coordinates has
$p\mid v_1\cdots v_{14}$.  For the primes in the second block of
Table~\ref{tab:gates14}, moreover, $J(14,p)=\varnothing$.
\end{proposition}

\begin{proof}
By Proposition~\ref{prop:basecomplete}, the initial search includes a
representative of every level-one orbit.  We examine every orbit left after
binary lifting at level $15$, missing no witness-free lift by
Lemma~\ref{lem:factor}.  Every such lift is either
proper by the gcd condition or satisfies Lemma~\ref{lem:neargcd} throughout
its integer residue class.  Lemma~\ref{lem:gate}(ii) gives the required
divisibility.  For the primes in the second block of
Table~\ref{tab:gates14}, the gcd test alone suffices, so every orbit is
eventually proper and $J(14,p)=\varnothing$.
\end{proof}

\section{Verified primes and proof of the main theorem}\label{sec:gateset}

The fourteen-runner computation was completed first, with the bound
of~\cite{MSS25} and $111$ primes; the first version of this paper,
arXiv:2609.02604v1, presented that proof alone.  The flag bound of
Section~\ref{sec:flag} was then developed for fifteen runners.  With this
bound, only $61$ of the $111$ verified primes are needed for thirteen
speeds, so no further prime computation is required.
We keep both proofs of $LRC(13)$: the first depends only on
published bounds, the second uses Theorem~\ref{thm:flag}.

\begin{proof}[First proof for fourteen runners]
Proposition~\ref{prop:prime13} gives $J(13,p)=\varnothing$ for $111$
primes whose logarithms sum to
$681.52920\ldots$.  Assuming $LRC(12)$, the finite-checking theorem of
Malikiosis, Santos, and Schymura~\cite{MSS25} gives
$\log(v_1\cdots v_{13})<670.34974\ldots$ for a primitive counterexample.

Any counterexample may be taken positive, primitive, and pairwise distinct:
a repeated coordinate leaves at most twelve distinct speeds, so $LRC(12)$
supplies a witness.  Hence
Lemma~\ref{lem:gate}(i) forces all $111$ primes to divide the speed product,
which is impossible.  Thus this route proves $LRC(13)$ without using the
lattice flag bound.
\end{proof}

For the second proof, let
\begin{align*}
\mathcal P_{13}={}&\{83,139,167,181,191\}
 \cup\{p:199\leq p\leq479,\ p\text{ prime}\}\\
 &\cup\{487,491,499,503,509,521,523,541,547\}.
\end{align*}
This set contains $\GateCountThirteen{}$ primes, whose logarithmic sum
exceeds the $n=13$ bound in Theorem~\ref{thm:flag}:
\begin{equation}\label{eq:mass13}
 \sum_{p\in\mathcal P_{13}}\log p
 =353.772559\ldots>\GateMassThirteen>341.031991,
\end{equation}

\begin{proof}[Second proof for fourteen runners]
The lower cases are given by~\cite{ST26}.  As above, sign changes,
rescaling, and $LRC(12)$ let us take a counterexample positive, primitive,
and pairwise distinct.  By Proposition~\ref{prop:prime13} and
Lemma~\ref{lem:gate}(i), every prime in $\mathcal P_{13}$ divides its speed
product.  Equation~\eqref{eq:mass13} contradicts Theorem~\ref{thm:flag},
so $LRC(13)$ holds.
\end{proof}

For fourteen speeds, Table~\ref{tab:gates14} lists the set
$\mathcal P_{14}$ of $\GateCountFourteen{}$ primes used in the proof.

\begin{table}[ht]
\centering
\caption{The primes used for $LRC(14)$.  Logarithms are natural and block
sums are rounded down.  The first block uses Lemma~\ref{lem:neargcd}; the
second verifies $J(14,p)=\varnothing$.}\label{tab:gates14}
\small
\begin{tabular}{@{}>{\raggedright\arraybackslash}p{0.21\textwidth}
                    >{\raggedright\arraybackslash}p{0.50\textwidth}
                    >{\raggedleft\arraybackslash}p{0.07\textwidth}
                    >{\raggedleft\arraybackslash}p{0.11\textwidth}@{}}
\toprule
Block & Primes & Count & $\sum\log p$\\
\midrule
$\tau(p)\leq12$ &
89, 131, 149, 157, 163, 167, 173, 179, 181, 191, 193, 197, 199,
211, 223, 227, 229, 233, 241 & 19 & 98.8199\\[2pt]
$\tau(p)\geq13$ &
Every prime from 239 through 569 except 241 & 52 & 310.0033\\
\midrule
\textbf{Total} & & \textbf{71} & \textbf{408.8233}\\
\bottomrule
\end{tabular}
\end{table}

The exact logarithmic sum satisfies
\begin{equation}\label{eq:mass14}
 \sum_{p\in\mathcal P_{14}}\log p
 >\GateMassFourteen>\TargetUB.
\end{equation}

\begin{proof}[Proof of Theorem~\ref{thm:main} for fifteen runners]
The first part of the proof gives $LRC(13)$.  Proposition
\ref{prop:prime14} supplies the prime divisibilities required in
Corollary~\ref{cor:mass}, and \eqref{eq:mass14} supplies their logarithmic
sum.  The corollary therefore excludes every counterexample to $LRC(14)$.
\end{proof}

\section{Verification and reproducibility}\label{sec:verification}

Each prime certificate records the covering search and node counts,
retained tuples at each binary level, and the final calculation.
Separate scripts check its arithmetic, file hashes, prime lists, and
logarithmic sums, and recompute the flag constants.  Logarithmic sums use
fifty-digit precision, with bounds rounded in the required direction.

At small primes, the two-branch search reproduces direct enumeration.
For $k=13$, our orbit counts agree with~\cite{STcode} at $p=43,83,$ and
$199$, after accounting for different normalizations.  Using or ignoring
the covering-number bound gives the same tuples at $p=223$; narrow and
wide integer encodings give the same decoded tuples on small test cases.
For $k=14$, general and specialized searches agree on fixed test cases.

At each binary lift, an independent recount confirmed the number of
retained tuples.  Small-case tests checked binary lifting and the
Chinese-remainder step.  Every completed $k=13$ branch is tested directly.
Rerunning both level-$14$ searches at $p=877$ reproduced their respective
logs byte for byte.  We also reproduced selected complete prime
computations for both values of $k$.

\begin{samepage}
We also checked the $LRC(12)$ input to both proofs.  In
the \texttt{for-k-12} snapshot of~\cite{STcode}, all $1{,}092$ residual
tuples across the $91$ primes used in~\cite{ST26} are equivalent to
$(1,2,\ldots,12)$ under permutations, sign changes, and multiplication by a
unit.  Proposition~4.4 of~\cite{ST26} eliminates them all.  The $LRC(9)$
input is available as
\texttt{results/result\_10} in the main snapshot of the same repository.
\par
\end{samepage}

The full computation has not been independently
reimplemented or formally verified.  The proof uses the completeness
statements above and the archived source and certificates.  The technical
records give file formats, per-prime counts, benchmarks, resource use,
unsuccessful runs, and program configurations.

\section{Concluding remarks}

Tightness forces the two orbits in~\eqref{eq:orbits13} to survive the binary
lifts; the same holds for $(1,\ldots,14)$ when $k=14$.  What the computation
adds is the converse: no other orbit survives for any of the $111$
archived $k=13$ primes, or for any of the verified $k=14$ primes with
$\tau(p)\geq13$.  This agrees with the finite-level picture suggested by
Proposition~7.1 of~\cite{ST26}; it is not an assertion about all primes.

\paragraph*{Acknowledgements.}
I thank Touch Sungkawichai and Tanupat Trakulthongchai for helpful comments
on earlier drafts and suggestions that improved the exposition of this
paper.  I also thank Tanupat Trakulthongchai for suggesting that the
fourteen- and fifteen-runner results be presented together.

\Needspace{9\baselineskip}
\section*{Data availability and use of artificial intelligence}

The code, certificates, and checker scripts are archived at
\url{https://doi.org/10.5281/zenodo.22066772} for fourteen runners and
\url{https://doi.org/10.5281/zenodo.22667683} for fifteen.
The author used large
language models to assist with code development, drafting, exploring proof
ideas, and searching for errors in the manuscript.  No model output is
treated as mathematical evidence.  The author checked the arguments, code,
and data and takes full responsibility.

\begin{samepage}

\end{samepage}

\end{document}